\documentclass[11pt,reqno,pdflatex]{amsart}
\usepackage{graphicx}
\usepackage{bbm}
\usepackage{url}
\usepackage{hyperref} 

\pdfoutput=1

\title{Rotor-router aggregation on the layered square lattice}
\author{Wouter Kager \and Lionel Levine}

\address{Wouter Kager, Department of Mathematics, VU University Amsterdam, De 
Boelelaan 1081, 1081\,HV Amsterdam, The Netherlands, {\normalfont
\url{http://www.few.vu.nl/~wkager}}}
\address{Lionel Levine, Department of Mathematics, Massachusetts Institute of 
Technology, Cambridge, MA 02139, {\normalfont 
\url{http://math.mit.edu/~levine}}}
\thanks{The second author was partly supported by a National Science 
Foundation Postdoctoral Fellowship.}

\date{March 19, 2010}

\newcommand{\Z}{\mathbbm{Z}}						
\newcommand{\N}{\mathbbm{N}}						
\newcommand{\arxiv}[1]{\href{http://arxiv.org/abs/#1}{arXiv:#1}}
\newcommand{\old}[1]{}

\newtheorem{theorem}{Theorem}
\newtheorem{lemma}[theorem]{Lemma}

\theoremstyle{remark}

\def\source{\mathtt{s}}
\def\target{\mathtt{t}}
\def\basis{\mathbf{e}}

\begin{document}

\begin{abstract}
	In rotor-router aggregation on the square lattice $\Z^2$, particles 
	starting at the origin perform deterministic analogues of random walks 
	until reaching an unoccupied site.  The limiting shape of the cluster of 
	occupied sites is a disk.  We consider a small change to the routing mechanism 
	for sites on the $x$- and $y$-axes, resulting in a limiting shape 
	which is a diamond instead of a disk.  We show that for a certain choice 
	of initial rotors, the occupied cluster grows as a perfect diamond.
\end{abstract}

\maketitle

\section{Introduction}
\label{sec:intro}

Recently there has been considerable interest in low-discrepancy deterministic 
analogues of random processes.  An example is rotor-router walk \cite{PDDK96}, 
a deterministic analogue of random walk.  Based at every vertex of the square 
grid~$\Z^2$ is a \emph{rotor} pointing to one of the four neighboring 
vertices.  A chip starts at the origin and moves in discrete time steps 
according to the following rule.  At each time step, the rotor based at the 
location of the chip turns clockwise 90~degrees, and the chip then moves to 
the neighbor to which that rotor points.

Holroyd and Propp~\cite{HP09} show that rotor-router walk captures the mean 
behavior of random walk in a variety of respects: stationary measure, hitting 
probabilities and hitting times. 
Cooper and Spencer~\cite{CS06} study rotor-router walks in which $n$~chips 
starting at arbitrary even vertices each take a fixed number~$t$ of steps, 
showing that the final locations of the chips approximate the distribution of 
a random walk run for $t$~steps to within constant error independent of $n$ 
and~$t$. Rotor-router walk and other low-discrepancy deterministic processes 
have algorithmic applications in areas such as broadcasting information in 
networks~\cite{DFS08} and iterative load-balancing~\cite{FGS10}.  The common 
theme running through these results is that the deterministic process captures 
some aspect of the mean behavior of the random process, but with significantly 
smaller fluctuations than the random process.

\emph{Rotor-router aggregation} is a growth model defined by repeatedly 
releasing chips from the origin $o \in \Z^2$, each of which performs a 
rotor-router walk until reaching an unoccupied site.  Formally, we set $A_0 = 
\{o\}$ and recursively define
\begin{equation}
	\label{eq:thegrowthrule}
	A_{m+1} = A_m \cup \{z_m\}
\end{equation}
for $m\geq 0$, where $z_m$ is the endpoint of a rotor-router walk started at 
the origin in~$\Z^2$ and stopped on exiting~$A_m$.  We do not reset the rotors when a new chip is released.

It was shown in~\cite{LP08,LP09} that for any initial rotor configuration,
the asymptotic shape of the set~$A_m$ is a Euclidean disk.  It is in some 
sense remarkable that a growth model defined on the square grid, and without 
any reference to the Euclidean norm $|x|=(x_1^2+x_2^2)^{1/2}$, nevertheless 
has a circular limiting shape.  Here we investigate the dependence of this 
shape on changes to the rotor-router mechanism.

The \emph{layered square lattice}~$\hat \Z^2$ is the directed multigraph 
obtained from the usual square grid~$\Z^2$ by reflecting all directed edges on 
the $x$- and $y$-axes that point to a vertex closer to the origin. For 
example, for each positive integer~$n$, the edge from $(n,0)$ to $(n-1,0)$ is 
reflected so that it points from $(n,0)$ to $(n+1,0)$.  Only edges on the $x$- 
and $y$-axes are affected.  Rotor-router walk on~$\hat \Z^2$ is equivalent to 
rotor-router walk on~$\Z^2$ with one modification: if the chip is on one of 
the axes, and the rotor points along the axis towards the origin after it is 
turned, then the chip ignores the rotor and moves in the opposite direction 
instead. 

For $n \geq 0$, let
\[ D_n = \left\{ (x,y) \in \Z^2 \,:\, |x|+|y|\leq n \right\}. \]
We call $D_n$ the \emph{diamond of radius~$n$}.  Our main result is the 
following.

\begin{theorem}
	\label{thm:diamond}
	There is a rotor configuration~$\rho_0$, such that rotor-router 
	aggregation $(A_m)_{m \geq 0}$ on~$\hat \Z^2$ with rotors initially 
	configured as~$\rho_0$ satisfies
	\[ A_{2n(n+1)} = D_n \qquad \text{ for all $n \geq 0$}. \]
\end{theorem}

A formal definition of rotor-router walk on~$\hat \Z^2$ and an explicit 
description of the rotor configuration~$\rho_0$ are given below.

Let us remark on two features of Theorem~\ref{thm:diamond}.  First, note that 
the rotor mechanism on~$\hat \Z^2$ is identical to that on~$\Z^2$ except for 
sites on the $x$- and $y$-axes.  Nevertheless, changing the mechanism on the axes
completely changes the limiting shape, transforming it from a disk into a 
diamond.  Second, not only is the aggregate close to a diamond, it is exactly 
equal to a diamond whenever it has the appropriate size 
(Figure~\ref{fig:simulation}). 

\begin{figure}
	\begin{center}
		\includegraphics[width=2.5in]{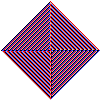}
	\end{center}
	\caption{The rotor-router aggregate of 5101 chips in the layered square lattice~$\hat \Z^2$ is a perfect diamond of radius~50. The colors encode the directions of the final rotors at the occupied vertices: red = north, blue = east, gray = south 
	and black = west.}
	\label{fig:simulation}
\end{figure}

In~\cite{KL10}, we studied the analogous stochastic growth model, known as 
\emph{internal DLA}, defined by the growth rule~\eqref{eq:thegrowthrule} using 
random walk on~$\hat \Z^2$.
This random walk has a \emph{uniform layering property}: at any fixed time, 
its distribution is a mixture of uniform distributions on the diamond layers
\[
	L_m = \left\{ (x,y) \in \Z^2 \,:\, |x|+|y| = m \right\},
	\qquad m \geq 1.
\]
It is for this reason that we call $\hat \Z^2$ the layered square lattice.

As a consequence of the uniform layering property, internal DLA on~$\hat \Z^2$ 
also grows as a diamond, but with random fluctuations at the boundary.  
Theorem~\ref{thm:diamond} thus represents an extreme of discrepancy reduction: 
passing to the deterministic analogue removes \emph{all} of the fluctuations 
from the random process, leaving only the mean behavior. For a similar ``no 
discrepancy'' result when the underlying graph is a regular tree instead 
of~$\hat \Z^2$, see~\cite{LL09}.

To formally define rotor-router walk on~$\hat \Z^2$, write $\basis_1 = (1,0)$, 
$\basis_2 = (0,1)$ and let $R = \left( \begin{smallmatrix} 0 & -1 \\ 1 & 0 
\end{smallmatrix} \right)$ be clockwise rotation by 90~degrees.  For each site $z\in \Z^2\setminus \{o\}$ there is a unique choice of a number $j \in 
\{0,1,2,3\}$ and a point~$w$ in the quadrant
\[	Q = \left\{ (x,y) \in \Z^2 \,:\, x\geq0,\, y>0 \right\}	\]
such that $z = R^jw$. Given $j$ and $w = (x,y)$, we associate to $z=R^jw$ a 
4-tuple $(e_z^0, e_z^1, e_z^2, e_z^3)$ of directed outgoing edges, where
\begin{equation}
	e_z^i = \begin{cases}
		(z,z+R^j\basis_2)		&	\text{if $i=2$ and $x=0$}; \\
		(z,z+R^{i+j}\basis_2)	&	\text{otherwise}.
	\end{cases}
	\label{eq:order}
\end{equation}
Thus, for $z\in Q$ (hence $j=0$ and $w=z$) the edges $e_z^0, e_z^1, e_z^2, 
e_z^3$ point respectively north, east, north, west when $z$ is on the 
$y$-axis; and north, east, south, west when $z$ is off the $y$-axis.  For~$z$ 
in another quadrant, the directions of $e_z^0, e_z^1, e_z^2, e_z^3$ are 
obtained using rotational symmetry.  To the origin we associate the 4-tuple 
$(e_o^0, e_o^1, e_o^2, e_o^3)$ where $e_o^i = (o,R^i\basis_2)$ for $i=0,1,2,3$.

For every $z\in \Z^2$, let $E_z$ be the multiset $\{ e_z^0, e_z^1, e_z^2, 
e_z^3 \}$. If $e = e_z^i \in E_z$, we denote by $e^+$ the next element 
$e_z^{i+1\bmod 4}$ of~$E_z$ under the cyclic shift. The \emph{layered square 
lattice}~$\hat \Z^2$ is the directed multigraph with vertex set $V = \Z^2$ and 
edge multiset $E = \bigcup_{z \in \Z^2} E_z$, where edges that appear twice 
in~$E_z$ have multiplicity two (Figure~\ref{fig:setup}, left). Thus every 
vertex has out-degree four, and every vertex except for the origin and its 
neighbors has in-degree four.

\begin{figure}
	\begin{center}
		\includegraphics{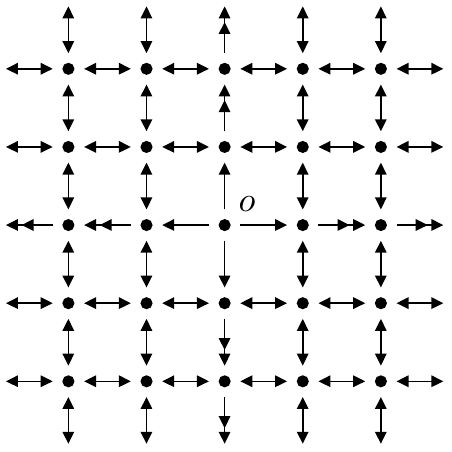} \qquad
		\includegraphics{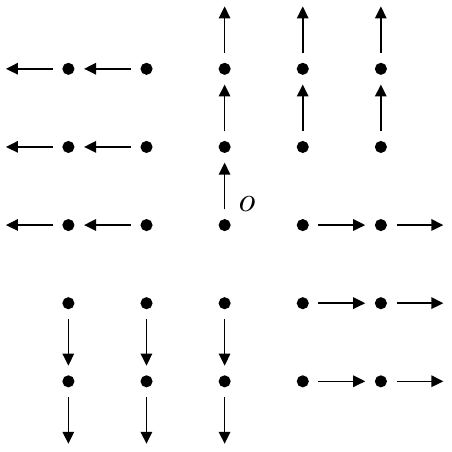}
	\end{center}
	\caption{Left: The layered square lattice~$\hat \Z^2$. Each directed edge 
	is represented by an arrow; multiple edges on the $x$- and $y$-axes are 
	represented by double arrows. The origin~$o$ is in the center. Right: The 
	initial rotor configuration~$\rho_0$.}
	\label{fig:setup}
\end{figure}

The initial rotor configuration~$\rho_0$ appearing in 
Theorem~\ref{thm:diamond} is given by
\begin{equation}
	\rho_0(z) = e_z^0,	\qquad z\in\Z^2.
	\label{eq:rho_0}
\end{equation}
It has every rotor in the quadrant~$Q$ pointing north, and rotor directions in 
the other quadrants given by rotational symmetry (Figure~\ref{fig:setup}, 
right).

We may now describe rotor-router walk on~$\hat \Z^2$ as follows.  Given a 
rotor configuration~$\rho$ with a chip at vertex~$z$, a single step of the 
walk consists of changing the rotor~$\rho(z)$ to~$\rho(z)^+$, and moving the 
chip to the vertex pointed to by the new rotor~$\rho(z)^+$.  This yields a new 
rotor configuration and a new chip location.  Note that if the walk visits~$z$ infinitely many times, then it visits all out-neighbors of~$z$ infinitely many times, and hence visits every vertex of~$\hat \Z^2$ (except for~$o$) infinitely many times.   It follows that rotor-router walk exits any finite subset of $\hat \Z^2$ in a finite number of steps; in particular, rotor-router aggregation terminates in finitely many steps. 

\old{
In the next section we generalize this formulation of the model to any finite directed multigraph, and prove a 
``strong abelian property'' of rotor-router walk. In 
Section~\ref{sec:shapetheorem} we apply the strong abelian property on the 
layered square lattice to prove Theorem~\ref{thm:diamond}.
}

\section{Strong Abelian Property}
\label{sec:abelian}

In this section we prove a ``strong abelian property'' of the rotor-router model, 
Theorem~\ref{thm:SAP}, which holds on any finite directed multigraph.  To prove Theorem~\ref{thm:diamond}, we will apply the results of this section to the induced subgraph~$D_n$ of~$\hat \Z^2$.  

Let $G = (V,E)$ be a finite directed multigraph (it may have loops and multiple 
edges).  Each edge $e\in E$ is directed from its source vertex~$\source(e)$ to 
its target vertex~$\target(e)$.  For a vertex $v \in V$, write
\[ E_v = \{ e \in E \,:\, \source(e)=v \} \]
for the multiset of edges emanating from~$v$.  The \emph{outdegree}~$d_v$ 
of~$v$ is the cardinality of~$E_v$.

Fix a nonempty subset $S \subset V$ of vertices called \emph{sinks}.  Let $V' 
= V\setminus S$, and for each vertex $v\in V'$, fix a numbering $e_v^0, 
\dotsc, e_v^{d_v-1}$ of the edges in~$E_v$. If $e = e_v^i \in E_v$, we denote 
by~$e^+$ the next element $e_v^{i+1\bmod d_v}$ of~$E_v$ under the cyclic 
shift.

A \emph{rotor configuration} on~$G$ is a function
\[ \rho: V' \to E \]
such that $\rho(v) \in E_v$ for all $v\in V'$. A \emph{chip configuration} 
on~$G$ is a function
\[ \sigma: V \to \Z. \]
Note that we do not require $\sigma\geq0$. If $\sigma(v) = m > 0$, we say 
there are $m$~chips at vertex~$v$; if $\sigma(v) = -m < 0$, we say there is a 
hole of depth~$m$ at vertex~$v$.

Fix a vertex $v \in V'$.  Given a rotor configuration~$\rho$ and a chip 
configuration~$\sigma$, the operation~$F_v$ of \emph{firing}~$v$ yields a new 
pair
\[ F_v(\rho,\sigma) = (\rho',\sigma') \]
where
\[
	\rho'(w) = \begin{cases}
		\rho(w)^+	&	\text{if $w=v$}; \\
		\rho(w)		&	\text{if $w\neq v$};
	\end{cases}
\]
and
\[
	\sigma'(w) = \begin{cases}
		\sigma(w)-1	&	\text{if $w=v$}; \\
		\sigma(w)+1	&	\text{if $w=\target(\rho(v)^+)$}; \\
		\sigma(w)	&	\text{otherwise}.
	\end{cases}
\]
In words, $F_v$ first rotates the rotor at~$v$, then sends a single chip 
from~$v$ along the new rotor~$\rho(v)^+$. We do not require $\sigma(v) > 0$ in 
order to fire~$v$. Thus if $\sigma(v) = 0$, i.e., no chips are present at~$v$, 
then firing~$v$ will create a hole of depth~1 at~$v$; if $\sigma(v) < 0$, so 
that there is already a hole at~$v$, then firing~$v$ will increase the depth 
of the hole by~1.

Observe that the firing operators commute: $F_vF_w = F_wF_v$ for all $v,w \in 
V'$. Denote by~$\N$ the nonnegative integers.  Given a function
\[ u: V' \to \N \]
we write
\[ F^u = \prod_{v\in V'} F_v^{u(v)} \]
where the product denotes composition.  By commutativity, the order of the 
composition is immaterial.

A rotor configuration~$\rho$ is \emph{acyclic} if the spanning subgraph 
$(V,\rho(V'))$ has no directed cycles or, equivalently, if for every nonempty 
subset $A \subset V'$ there is a vertex $v \in A$ such that $\target(\rho(v)) 
\notin A$.

In the following theorem and lemmas, for functions~$f,g$ defined on a set of 
vertices $A \subset V$, we write ``$f = g$ on~$A$'' to mean that $f(v) = g(v)$ for 
all $v \in A$, and ``$f \leq g$ on~$A$'' to mean that $f(v) \leq g(v)$ for all $v 
\in A$.

\begin{theorem}[Strong Abelian Property]
	\label{thm:SAP}
	Let~$\rho$ be a rotor configuration and~$\sigma$ a chip configuration 
	on~$G$.	Given two functions $u_1,u_2: V'\to \N$, write
	\[ F^{u_i}(\rho,\sigma) = (\rho_i,\sigma_i), \qquad i=1,2. \]
	If $\sigma_1 = \sigma_2$ on~$V'$, and both $\rho_1$ and~$\rho_2$ are 
	acyclic, then $u_1 = u_2$.
\end{theorem}

Note that the equality $u_1 = u_2$ implies that $\rho_1 = \rho_2$, and that 
$\sigma_1 = \sigma_2$ on all of~$V$. 
For a similar idea with an algorithmic application, 
see~\cite[Theorem~1]{FL10}.

In a typical application of Theorem~\ref{thm:SAP}, we take $\sigma_1 = 
\sigma_2 = 0$ on~$V'$, and $u_1$ to be the usual rotor-router odometer 
function
\[ u_1(v) = \#\{1 \leq j \leq k \,:\, v_j = v\} \]
where $v_1, v_2, \dotsc, v_k$ is a \emph{complete legal firing sequence} for 
the initial configuration~$(\rho,\sigma)$; that is, a sequence of vertices 
which, when fired in order, causes all chips to be routed to the sinks without 
ever creating any holes. Provided $u_1(v) > 0$ for all $v\in V'$, the 
resulting rotor configuration~$\rho_1$ is acyclic: indeed, for any nonempty 
subset $A \subset V'$, the rotor at the last vertex of~$A$ to fire points to a 
vertex not in~$A$.  

The usual abelian property of rotor-router walk~\cite[Theorem~4.1]{DF91} says 
that any two complete legal firing sequences have the same odometer function.  
The Strong Abelian Property allows us to drop the hypothesis of legality: any 
two complete firing sequences whose final rotor configurations are acyclic 
have the same odometer function, \emph{even if one or both of these firing 
sequences temporarily creates holes}.

In our application to rotor-router aggregation on the layered square lattice, 
we take $V=D_n$ and $S=L_n$.  We will take $\sigma$ to be the chip 
configuration consisting of $2n(n+1)+1$ chips at the origin, and $\rho$ to be 
the initial rotor configuration~$\rho_0$.  Letting the chips at the origin in 
turn perform rotor-router walk until finding an unoccupied site defines a 
legal firing sequence (although not a complete one, since not all chips reach the sinks).  In the next section, we give an explicit formula for the corresponding odometer function, and use Theorem~\ref{thm:SAP} to prove its correctness.  The proof of Theorem~\ref{thm:diamond} is completed by showing that each nonzero vertex in $D_n$ receives exactly one more chip from its neighbors than the number of times it fires. 

To prove Theorem~\ref{thm:SAP} we start with the following lemma.

\begin{lemma}
	\label{lem:mustcycle}
	Let $u: V'\to \N$, and write
	\[ F^u(\rho,\sigma) = (\rho_1,\sigma_1). \]
	If $\sigma = \sigma_1$, and $\rho_1$ is acyclic, then $u=0$.
\end{lemma}

\begin{proof}
	Let $A = \{ v\in V': u(v)>0 \}$, and suppose that~$A$ is nonempty.  Since 
	$\rho_1$ is acyclic, there is a vertex $v\in A$ whose rotor $\rho_1(v)$ 
	points to a vertex not in~$A$. The final time $v$ is fired, it sends a 
	chip along this rotor; thus, at least one chip exits~$A$. Since the 
	vertices not in~$A$ do not fire, no chips enter~$A$, hence
	\[ \sum_{v\in A} \sigma_1(v) < \sum_{v\in A} \sigma(v) \]
	contradicting $\sigma = \sigma_1$. Therefore, $A$ is empty.
\end{proof}

Theorem~\ref{thm:SAP} follows immediately from the next lemma.

\begin{lemma}
	\label{lem:fireless}
	Let $u_1,u_2: V'\to \N$, and write
	\[ F^{u_i}(\rho,\sigma) = (\rho_i,\sigma_i), \qquad i=1,2. \]
	If $\rho_1$ is acyclic and $\sigma_2 \leq \sigma_1$ on~$V'$, then $u_1 
	\leq u_2$ on~$V'$.
\end{lemma}

\begin{proof}
	Let
	\[ (\hat\rho,\hat\sigma) = F^{\min(u_1,u_2)}(\rho,\sigma). \]
	Then $(\rho_1,\sigma_1)$ is obtained from $(\hat\rho,\hat\sigma)$ by 
	firing only vertices in the set $A = \{ v\in V': u_1(v) > u_2(v) \}$, so
	\[ \hat\sigma \leq \sigma_1 \text{ on $A^c$}. \]
	Likewise, $(\rho_2,\sigma_2)$ is obtained from $(\hat\rho,\hat\sigma)$ by 
	firing only vertices in~$A^c$, so
	\[ \hat\sigma \leq \sigma_2 \leq \sigma_1 \text{ on $A$}. \]
	Thus $\hat\sigma \leq \sigma_1$ on~$V$.  Since $\sum_{v \in V} 
	\hat\sigma(v) = \sum_{v \in V} \sigma_1(v)$ it follows that
	$\hat\sigma = \sigma_1$. Taking
	\[ u = u_1 - \min(u_1,u_2) \]
	in Lemma~\ref{lem:mustcycle}, since $F^u(\hat\rho, \hat\sigma) = (\rho_1, 
	\sigma_1)$ we conclude that $u=0$.
\end{proof}

\section{Proof of Theorem~\ref{thm:diamond}}
\label{sec:shapetheorem}

Consider again the rotor-router model on the layered square 
lattice~$\hat\Z^2$.  We will work with the induced subgraph~$D_n$ 
of~$\hat\Z^2$, taking the sites in the outermost layer~$L_n$ as sinks.

Recall our notation
\[ Q = \left\{ (x,y)\in \Z^2 \,:\, x\geq0, y>0 \right\} \]
for the first quadrant of~$\Z^2$.  We have $\Z^2 = \{o\} \cup \bigl( 
\bigcup_{i=0}^3 R^iQ \bigr)$, where $R = \left( \begin{smallmatrix} 0 & -1 \\
1 & 0 \end{smallmatrix} \right)$ is clockwise rotation by $90$ degrees.  
Fix~$n$, and for $z = (x,y) \in D_n$ write
\[ \ell_z = n-|x|-|y|. \]
Consider the sets
\[
	C_2 = \left\{ (x,y)\in Q \cap D_{n-1} \,:\, x>0, \,  y \geq 2, \, \ell_{(x,y)} \equiv 2 
	\bmod 4 \right\}
\]
\[
	C_3 = \left\{ (x,y)\in Q \cap D_{n-1} \,:\, x>0, \,  y \geq 1, \, \ell_{(x,y)} \equiv 3 
	\bmod 4 \right\}
\]
and
\[ C = \bigcup_{i=0}^3 R^i (C_2 \cup C_3). \]

Define $u_n : D_{n-1} \to \N$ by 
\begin{equation}
	\label{eq:u_n}
	 u_n = u'_n - 1_C
\end{equation}
where
\begin{equation}
	\label{eq:u'_n}
	u'_n(z) = \begin{cases}
		2n(n+1)			&	\text{if $z=o$};\\
		\ell_z (\ell_z+1)		&	\text{if $z\neq o$};
	\end{cases}
\end{equation}
and $1_C(z)$ is the indicator function which is $1$ for $z \in C$ and $0$ 
for $z \notin C$.

Let~$\rho_0$ be the initial rotor configuration~\eqref{eq:rho_0}, and define 
the rotor configuration~$\rho_n$ on~$D_{n-1}$ and chip 
configuration~$\sigma_n$ on~$D_n$ by setting
\[  F^{u_n}(\rho_0,(2n^2+2n+1)\delta_o) = (\rho_n,\sigma_n).  \]
From the formula~\eqref{eq:u_n} it is easy to obtain an explicit description 
of~$\rho_n$, and to verify that these rotor configurations are acyclic for all 
$n\geq1$. Figure~\ref{fig:RotorConfig} depicts the rotor configurations 
$\rho_2, \rho_3, \dotsc, \rho_7$ in the first quadrant. 

\begin{figure}
	\begin{center}
		\includegraphics{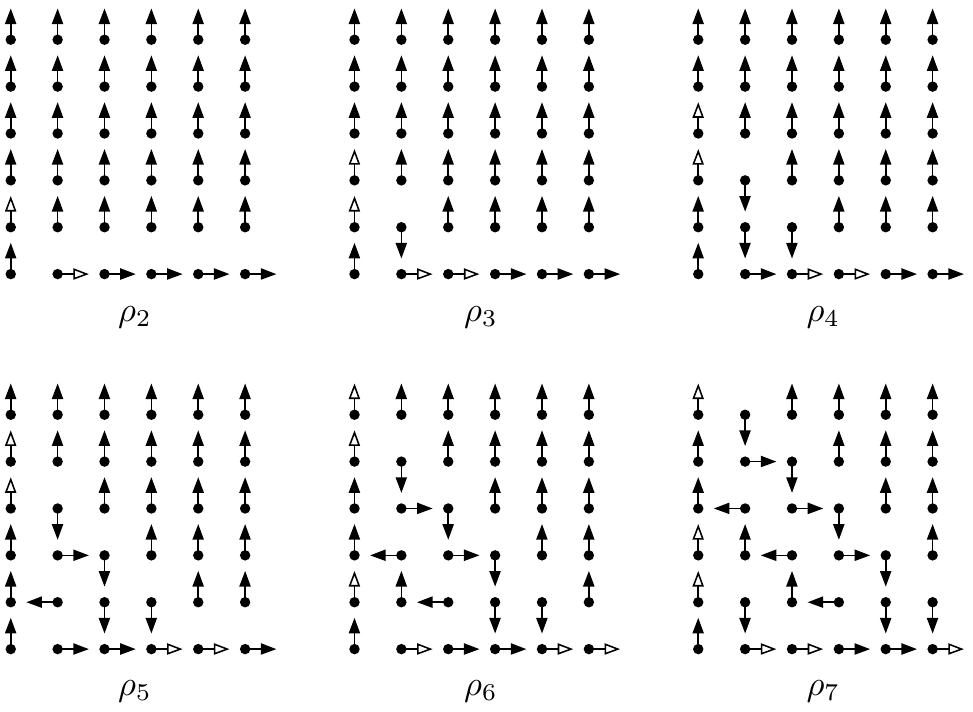}
	\end{center}
	\caption{The rotor configurations $\rho_2, \rho_3, \dotsc, \rho_7$ in the 
	first quadrant. The lower left corner is the origin in each picture. On 
	the axes, the black arrows correspond to the directed edge~$e_z^0$ 
	in~\eqref{eq:order}, and open-headed arrows to~$e_z^2$.}
	\label{fig:RotorConfig}
\end{figure}

\begin{lemma}
	\label{lem:diamond}
	For all $n\geq1$, we have $\sigma_n = 1_{D_n}$.
\end{lemma}

\begin{proof}
	The origin~$o$ has no incoming edges in $\hat \Z^2$, so it receives no 
	chips from its neighbors.  Since $u_n(o)=2n^2+2n$, the origin is left with 
	exactly one chip after firing. 	The sink vertices~$L_n$ do not fire and 
	only receive chips.  Since $u_n(z) = 2$ for all $z\in L_{n-1}$, it follows 
	from~\eqref{eq:order} and~\eqref{eq:rho_0} that exactly one chip is sent 
	to each sink vertex.

	It remains to show that $\sigma_n(z) = 1$ for each vertex $z\in D_{n-1} 
	\setminus \{o\}$, i.e., that the number of chips sent to~$z$ by its 
	neighbors is one more than the number of times~$z$ is fired (that is, 
	$1+u_n(z)$). To show this, write
	\[  F^{u'_n}(\rho_0,(2n^2+2n+1)\delta_o) = (\rho'_n,\sigma'_n)	\]
	where $u'_n$ is given by~(\ref{eq:u'_n}).
	We will argue that $\sigma'_n(z) = 1$ and that $\sigma_n(z) = 
	\sigma'_n(z)$.  By symmetry, it suffices to consider points $z = (x,y)$ in 
	$D_{n-1}\cap Q$.  We argue separately in the two cases $x=0$ and $x>0$ (on 
	the axis and off the axis).

	\textbf{Case 1:} $x=0$.
	Under $F^{u'_n}$, the site $z$ fires $\ell_z (\ell_z+1)$ times. If $y>1$, its 
	neighbor $z-\basis_2$ fires $(\ell_z+1) (\ell_z+2)$ times, and from 
	\eqref{eq:order} and~\eqref{eq:rho_0} we see that it sends a chip to~$z$ 
	every even time it is fired. Since $(\ell_z+1) (\ell_z+2)$ is even, it 
	follows that $z-\basis_2$ sends $\tfrac12 (\ell_z+1) (\ell_z+2)$ chips 
	to~$z$. The same is true if $y=1$, since then $\ell_z = n-1$, and the 
	origin $o=z-\basis_2$ sends $\tfrac12 n(n+1)$ chips to~$z$.
	
	The only other vertices that send chips to~$z$ under $F^{u'_n}$ are its 
	left and right neighbors $z\pm \basis_1$. Since $\ell_{z\pm\basis_1} = 
	\ell_z-1$, these neighbors fire $\ell_z (\ell_z-1)$ times. We claim that 
	together they send $\tfrac12 \ell_z (\ell_z-1)$ chips to~$z$. To see this,  
	note that if we fire these two vertices in parallel, they send one chip 
	to~$z$ every two times we fire. We therefore conclude that
	\[
		\sigma'_n(z) = 	\tfrac12 (\ell_z+1) (\ell_z+2) + \tfrac12 \ell_z 
		(\ell_z-1)  - \ell_z (\ell_z+1) = 1.
	\]

	To show that $\sigma_n(z) = \sigma'_n(z)$, note first that neither $z$ 
	nor~$z-\basis_2$ is in~$C$ because $x=0$. The right neighbor $z+\basis_1$ 
	might be in~$C$, but since $\ell_z (\ell_z-1)$ is even, the last chip sent 
	from $z+\basis_1$ by~$F^{u'_n}$ does not move to~$z$. The left neighbor 
	$z-\basis_1$ is in~$C$ only if $\ell_{z-\basis_1} = \ell_z-1 \equiv 3\bmod 
	4$, which implies $\ell_z (\ell_z-1) \equiv 0\bmod 4$. Hence if 
	$z-\basis_1$ is in~$C$, the last chip sent from $z-\basis_1$ by~$F^{u'_n}$ 
	moves west. It follows that $F^{u_n}$ and~$F^{u'_n}$ fire~$z$ the same 
	number of times and send the same number of chips to~$z$, hence 
	$\sigma_n(z) = \sigma'_n(z) = 1$.

	\textbf{Case 2:} $x>0$.
	To argue that $\sigma'_n(z) = 1$, as an initial step we \emph{un}fire 
	every vertex on the positive $x$-axis $B = \{ (m,0) \in\Z^2: m>0 \}$ once. 
	Since all initial rotors on~$B$ point east, this turns all these rotors 
	north without affecting the number of chips at~$z$ (nor at any other 
	vertex of~$Q$).

	Now we apply~$F^{u'_n}$. By firing the four neighbors of~$z$ in parallel, 
	it is easy to see from~\eqref{eq:order} that they send one chip to~$z$ 
	every firing round, since after every round exactly one of their rotors 
	points to~$z$. Hence, firing these neighbors $\ell_z (\ell_z-1)$ times 
	each sends $\ell_z (\ell_z-1)$ chips to~$z$. Since $\ell_{z-\basis_1} = 
	\ell_{z-\basis_2} = \ell_z+1$, the two neighbors $z-\basis_1$ and 
	$z-\basis_2$ each fire
	\[ (\ell_z+1)(\ell_z+2) - \ell_z(\ell_z-1) = 4\ell_z+2 \]
	additional times under~$F^{u'_n}$.  Considering what happens when they are 
	fired in parallel shows that they send one chip to~$z$ every two times 
	they fire, meaning that $2\ell_z+1$ additional chips are sent to~$z$.
	
	Finally, to obtain~$\sigma'_n$ we must fire every vertex in~$B$ once more. 
	But since $F^{u'_n}$ fires each vertex in~$B$ an even number of times, 
	their rotors are now pointing either north or south, so firing them once 
	more does not affect the number of chips at~$z$. Hence
	\[	\sigma'_n(z) = \ell_z(\ell_z-1) + (2\ell_z+1) - \ell_z(\ell_z+1) = 1. 
	\]

	To finish the proof, we now argue that $\sigma_n(z) = \sigma'_n(z)$. First 
	note that since $\ell_v (\ell_v+1)$ is even for all $v\in D_{n-1}$, it 
	follows from~\eqref{eq:order} that the last chips sent from~$z\pm 
	\basis_1$ by~$F^{u'_n}$ do not move to~$z$. However, consider the neighbor 
	$z+\basis_2$. If $\ell_{z+\basis_2} = \ell_z-1 \equiv 3\bmod 4$, its final 
	rotor points north after firing $\ell_z (\ell_z-1)$ times, while if 
	$\ell_z-1 \equiv 2\bmod 4$, its final rotor points south. It therefore 
	follows from the definition of~$C$, that $F^{u_n}$ sends one fewer chip 
	from $z+\basis_2$ to~$z$ than $F^{u'_n}$ in case $\ell_z \equiv 3\bmod 4$ 
	and $y+1 \geq 2$.  Likewise, $F^{u_n}$ sends one fewer chip from 
	$z-\basis_2$ to~$z$ than $F^{u'_n}$ in case $\ell_z \equiv 2\bmod 4$ and 
	$y-1 \geq 1$.  But these are precisely the two cases when $z\in C$, hence 
	$F^{u_n}$ also fires~$z$ once fewer than $F^{u'_n}$. Therefore, 
	$\sigma_n(z) = \sigma'_n(z) = 1$.
\end{proof}

We remark that the rotor configuration $\rho_n$ is obtained from $\rho'_n$ by \emph{cycle-popping}: that is, for each directed cycle of rotors in $\rho'_n$, unfire each vertex in the cycle once.  Popping a cycle causes each vertex in the cycle to send one chip to the previous vertex, so there is no net movement of chips.  Let $\rho''_n$ be the acyclic rotor configuration obtained from cycle-popping, and let
	\[ u''_n = u'_n - c_n \]
where $c_n(z)$ is the number of times~$z$ is unfired during cycle-popping.  Then
	\[ F^{u''_n}(\rho_0, (2n^2+2n+1)\delta_o) = (\rho''_n, 1_{D_n}). \]
By Lemma~\ref{lem:diamond}, we have
	\[ F^{u_n}(\rho_0, (2n^2+2n+1)\delta_o) = (\rho_n, 1_{D_n}). \]
By the Strong Abelian Property (Theorem~\ref{thm:SAP}), it follows that $u''_n=u_n$.  In particular, $c_n = 1_C$ and $\rho''_n = \rho_n$. 

\begin{proof}[Proof of Theorem~\ref{thm:diamond}]
	For $m\geq0$, let $r_m$ be the smallest integer such that $2r_m(r_m+1) > 
	m$. Consider a modified rotor-router aggregation defined by the growth 
	rule
	\[ \tilde A_{m+1} = \tilde A_m \cup \{\tilde z_m\} \]
	where $\tilde z_m$ is the endpoint of a rotor-router walk in $\hat \Z^2$ and 
	stopped on exiting $\tilde A_m \cap D_{r_m-1}$. Define $\tilde u_n: D_{n-1} \to \N$ by setting
	\[
		\tilde u_n(z) = \# \text{ of times $z$ fires during the formation of 
		$\tilde A_{2n(n+1)}$}.
	\]
	We will induct on~$n$ to show that $u_n = \tilde u_n$ for all $n\geq1$. Since 
	$u_n = \tilde u_n$ implies $A_{2n(n+1)} = \tilde A_{2n(n+1)} = D_n$ by 
	Lemma~\ref{lem:diamond}, this proves the theorem.

	The base case of the induction is immediate: $u_1 = \tilde u_1 = 4\delta_o$. For 
	$n\geq 2$, in the induced subgraph~$D_n$ of~$\hat \Z^2$ with sink 
	vertices~$L_n$ we have
	\[	F^{u_n}(\rho_0,(2n^2+2n+1)\delta_o) = (\rho_n,1_{D_n}) \]
	by Lemma~\ref{lem:diamond}. On the other hand,
	\[	F^{\tilde u_n}(\rho_0,(2n^2+2n+1)\delta_o) = (\tilde \rho_n, \tilde \sigma_n)	\]
	for some rotor configuration $\tilde \rho_n$ on~$D_{n-1}$ and chip 
	configuration~$\tilde \sigma_n$ on~$D_n$.  By the inductive hypothesis, $\tilde A_{2n(n-1)} = D_{n-1}$, from which it follows that in the formation of $\tilde A_{2n(n+1)}$ 
	from~$\tilde A_{2n(n-1)}$, all rotor-router walks are stopped on 
	exiting~$D_{n-1}$.  Together these facts imply that $\tilde \sigma_n = 1$ on~$D_{n-1}$.  Moreover, since $\rho_0$ is acyclic, $\tilde \rho_n$ is acyclic (each rotor points in the direction a chip last exited). 
	The Strong Abelian Property (Theorem~\ref{thm:SAP}) now gives $u_n = \tilde u_n$, which completes the inductive step.
\end{proof}

\end{document}